\documentclass[11pt,leqno]{amsart}

\usepackage{epsfig}
\usepackage[utf8]{inputenc}

\usepackage{amsthm}

 \usepackage{mathtools}
\mathtoolsset{showonlyrefs}

 \usepackage[T1]{fontenc}

 \usepackage{amssymb}

    \newtheorem{theorem}{Theorem}
    
    \newtheorem{proposition}{Proposition}

    \newtheorem{remark}{Remark}
    \newtheorem{corollary}{Corollary}
    
    \newtheorem{definition}{Definition}
\newcommand{\C}{\mathcal{C}}

\newcommand{\D}{\mathcal{D}}

\newcommand{\G}{\mathcal G}

    \newcommand\CC{\hbox{C\kern -.58em {\raise .54ex \hbox
    {$\scriptscriptstyle |$}}
    \kern-.55em {\raise .53ex \hbox{$\scriptscriptstyle |$}} }}
      \newcommand\qd{\hfill$\sqcap\kern-8.0pt\hbox{$\sqcup$}$}
    \newcommand\NN{\hbox{I\kern-.2em\hbox{N}}}
       \newcommand\nn{\hbox{I\kern-.2em\hbox{N}}}
    \newcommand\RR{I\!\!R}
    \newcommand\sRR{{\sl \hbox{I\kern-.2em\hbox{R}}}}
    \newcommand\QQ{\hbox{I\kern-.53em\hbox{Q}}}
     
\newcommand\sign{\hbox{Sign}}

\newcommand\signp{{\hbox{Sign}^+}}
    
 \newcommand\spo{\hbox{Sign}_0^+}
 \newcommand \smo{\hbox{Sign}_0^-}
    
    \newcommand\R{\displaystyle {\mathcal R}}
    
  \everymath{\displaystyle}

\usepackage{stackengine}
\usepackage{scalerel}
\usepackage{xcolor,amssymb}

\begin{document}
	 
\title[Mathematical Study of Reaction-Diffusion in Congested Crowd Motion]
{Mathematical Study of Reaction-Diffusion in Congested Crowd Motion}
 \newcommand{\cs}{$^\dagger$} \newcommand{\cm}{$^\ddagger$}

\author[N. Igbida]{Noureddine Igbida\cs}
\author[F. Karami]{Fahd Karami\cm}
\author[D. Meskine]{Driss Meskine\cm}

\thanks{\cs Institut de recherche XLIM-DMI, UMR-CNRS 6172, Facult\'e des Sciences et Techniques, Universit\'e de Limoges, France. Email:   noureddine.igbida@unilim.fr}
\thanks{\cm  Laboratoire de Mathématique Informatique et modélisation des Systèmes Complèxes, EST Essaouira,  Cadi Ayyad University 
	Marrakesh,   Morocco.   Emails:   fa.karami@uca.ac.ma, dr.meskine@uca.ac.ma}


\newcommand{\nfont}{\fontshape{n}\selectfont}


\date{\today}

\begin{abstract}
	This paper establishes existence, uniqueness, and an $L^1-$comparison principle for weak solutions of a PDE system modeling phase transition reaction-diffusion in congested crowd motion.  We consider a general reaction term and mixed homogeneous (Dirichlet and Neumann) boundary conditions. This model is applicable to various problems, including multi-species diffusion-segregation and pedestrian dynamics with congestion.  Furthermore, our analysis of the reaction term yields sufficient conditions combining the drift with  the reaction that guarantee the absence of congestion, reducing the dynamics to a constrained linear reaction-transport equation.  \end{abstract}

\maketitle

\section{Introduction}
\setcounter{equation}{0}
 
The study of crowd motion has drawn significant inspiration from fluid dynamics, with nonlinear partial differential equations (PDEs) serving as fundamental tools for modeling. Over the past two decades, substantial progress has been made in understanding congestion dynamics using nonlinear PDEs. This paper explores a specific class of models that incorporate reaction terms to analyze congestion in crowd motion.

The starting point  of our mathematical framework is the linear transport equation:
\begin{equation} \label{transport0}
	\partial_t  u + \nabla \cdot ( u\:  U) = g,
\end{equation} 
where $ u= u(t,x)$ represents the density of individuals at time $t \geq 0$ and position $x \in \mathbb{R}^N$ ($N=2$). To ensure that the carrying capacity of the environment is not exceeded, we impose the constraint:
\begin{equation}\label{constraints0}
	0 \leq  u \leq 1.
\end{equation}  
The velocity field $U$ dictates the collective movement of individuals. Despite numerous proposed models for $U$, no universally accepted formulation exists.

The source term \(g\) in \eqref{transport0} is introduced to account for the appearance or disappearance of individuals within the domain. For instance, when \(u\) models the active population, a central role of \(g\) is to incorporate absorption terms that represent the loss of active individuals, either because they leave the domain or because their motion slows down under congestion effects.  

A typical choice may be given by  
\begin{equation}\label{example1}
	 g = -\alpha\, u \,(u - u_{\text{eq}}), 
\end{equation}
which models the tendency of individuals to align with the local crowd density. In this case, movement slows down in high-density regions due to congestion. Here, \(u_{\text{eq}}\) denotes an equilibrium density, while \(\alpha > 0\) is a parameter controlling the rate of relaxation toward equilibrium. The results presented in the last section illustrate how this setting can be treated as a special case of the general framework.  

Another possibility is  
\[
g = -\alpha\, u^2,
\]  
which models density-dependent slowdown: as density increases, pedestrian motion naturally decelerates, ensuring reduced flow in crowded regions.  

To describe \emph{fatigue effects} over long time horizons, one may instead consider  
\[
g = -\lambda\, u\, e^{-\beta t},
\]  
which captures the gradual reduction in movement caused by fatigue.  

Finally, \(g\) can also represent \emph{external factors} influencing crowd density, such as the influx of spectators entering a stadium or the outflow of individuals during an evacuation. In such cases, one may set  
\[
g(x,t) = \beta\, f(x,t),
\]  
where \(\beta > 0\) is the influx (or outflux) rate and \(f(x,t)\) prescribes the spatial and temporal distribution of the entry or exit points.  

\medskip 
The main challenge of the model \eqref{transport0}-\eqref{constraints0} in congested crowd modeling is defining a velocity field $U$ that balances macroscopic crowd dynamics, such as reaching a target or avoiding obstacles, with individual behaviors, including speed variations, avoidance strategies, and attraction to crowds.   While maintaining the primary goal  of  move toward a designated target along a predefined trajectory guided by a velocity vector $V=V(t,x)$, we anticipate a deviation flow $W$  that accounts for local interactions, capturing avoidance and attraction behaviors based on the distribution of nearby pedestrians.  The correction flux $W$ enable  the crowd dynamically adjusts trajectories to mitigate congestion.  Specifically, we express the velocity field as:
\begin{equation}\label{decomposeU}
	U=V+ W,
\end{equation}   
 A well-known example of this framework    was first introduced in \cite{MRS1} within the context of gradient flow in the Wasserstein space of probability measures (see also \cite{MRS1,MRS2,MS}).    In  this framework, the vector field  $W=W[ u]$  is determined by an unknown potential $p$ supported in the congested region $[ u=1$, yielding the formulation:
 \begin{equation}
 	W= - \nabla p, \quad \text{where } p \geq 0 \text{ and } p( u-1)=0.
 \end{equation} 
 In some sense this formulation tends  to account for the random Brownian-like movement of individuals at a microscopic level,  to addressing potential congestion.   This yields  to the formulation:
 \begin{equation}\label{W}
 	U =V-\nabla p, \quad \text{where } p \geq 0 \text{ and } p\: ( u-1)=0.
 \end{equation} 
    Other version of corrections  $W$ are  proposed in \cite{EIJ}  (see also \cite{IgUr}) within the context of minimum flow problem. For instance, one can  take  the term $W= -\vert  \nabla p\vert^{p-2}  \nabla p $ to describe the case where   the crowd behaves like a non-Newtonian shear thickening or dilatant fluid, with the viscosity depending on the shear stress. To take $p\to \infty$   simulates the movement of individuals like grains in a sandpile   to alleviate congestion.  
  
   \medskip 
 Keeping in mind the representative case \eqref{W},  the model we have in mind falls into the scope of the nonlinear diffusion-transport equation  of the type
\begin{equation}\label{pdetype}
	\left\{
	\begin{array}{l}
		\displaystyle \frac{\partial u }{\partial t}  -\Delta p +\nabla \cdot (u  \: V)= g   \\    \\
		\displaystyle u \in \sign(p)\end{array}\right.  \quad  \hbox{ in } Q:= (0,T)\times \Omega,
\end{equation} 
   Here,   $\Omega\subset\RR^N$ be a bounded   open set,  $\sign$ is the maximal monotone graph defined in $\RR$ by
$$\sign(r)= \left\{ \begin{array}{ll}
	\displaystyle 1 &\hbox{ for any }r>0\\  
	\displaystyle [-1,1]\quad & \hbox{ for } r=0\\
	\displaystyle -1 &\hbox{ for any }r<0, \end{array}\right. $$
$V$ and $g$ are the given  velocity field and the reaction term respectively, satisfying besides assumptions we precise next.  
 Observe that no sign condition is imposed on \(u\). In this case, the PDE may describe segregation phenomena, where the positive and negative parts of \(u\) correspond to two distinct populations subject to segregation effects. Moreover,  the reaction term can be chosen in the form  
\[
g(.,u) = g_1(\cdot, u^+) + g_2(\cdot, u^-),
\]  
where \(u^+\) and \(u^-\) denote, respectively, the positive and negative parts of \(u\).

In the case of nonnegative solution (one phase problem), the problem may be written in the widespread  form
\begin{equation} \label{pdetype+}
	\left\{
	\begin{array}{l}
		\displaystyle \frac{\partial u }{\partial t}  -\Delta p +\nabla \cdot (u  \: V)= g  \\    \\
		\displaystyle  0\leq u\leq 1,\:  p\geq 0,\:  p(u-1)=0 \end{array}\right.  \quad  \hbox{ in } Q.
\end{equation} 
Following existence, uniqueness and $L^1-$contraction approach  of the work \cite{Igshaw}, this  paper explores these mathematical dynamic \eqref{pdetype}, focusing on existence, uniqueness, and stability properties of solutions. We consider \eqref{pdetype} subject to 

\begin{equation}\label{BC0}
	\left\{  
		\begin{array}{ll}
		\displaystyle p= 0  & \hbox{ on }\Sigma_D:= (0,T)\times \Gamma_D\\  \\
		\displaystyle (\nabla p- u  \: V)\cdot \nu = 0  & \hbox{ on }\Sigma_N:= (0,T)\times \Gamma_N  \end{array} \right.
\end{equation}  
and initial data
$$	\displaystyle  u (0)=u _0 \quad \hbox{ in }\Omega.$$ 
Here  $\Gamma_D$ and $ \Gamma_N$ is a partition of the boundary $\partial \Omega$ assume  to be     regular    of class $\C^2.$  

\medskip 
 Degenerate parabolic problems like \eqref{pdetype} have seen considerable progress in recent decades, notably due to  Carillo's work \cite{Ca} extending the Kruzhkov doubling of variables technique.  While  Carillo's weak-entropy solution concept and related techniques (see \cite{AnIgSurvey}) are powerful, they primarily address cases with space-independent $V$ and homogeneous Dirichlet boundary conditions.  Although the entropic solution is suitable for general problems, the uniqueness of weak solutions and their $L^1-$comparison principle for linear drift cases remains a challenging question.  
 
 \medskip
 Despite the second-order term, the first-order term involving the vector field $V$ is crucial.  When $p\equiv 0,$  the PDE reduces to a linear initial-boundary value problem for the continuity equation.  Uniqueness is generally expected under assumptions like bounded total variation coefficients for $V. $  The regularity of $V$  and the treatment of boundary conditions are intertwined with uniqueness proofs.  The direction of $V$ on the boundary influences the handling of flux traces (cf. \cite{Ambrosio, ACM, Anzellotti, ChenFrid, ChenTZ, AnBo, AnIg2}). The concept of renormalized solutions (cf. \cite{DiLions, Boyer, BoyerFabrie, Mischler, Ambrosio, AC, ACM, CDS, Figali, LebrisLions} ) has been instrumental in addressing these issues. The recent \cite{Igshaw} paper demonstrates how to use this approach with a Hele-Shaw type second-order term to prove weak solution uniqueness for mixed Dirichlet-Neumann boundary conditions with outward-pointing $V.$   
 
 \medskip 
 As far as we know, existence and uniqueness with a reaction term $g=g(t,x,u)$ remain open. This paper addresses this gap under general assumptions on $g$ (including locally Lipschitz dependence on $u$).  Leveraging the $L^1-$comparison principle from \cite{Igshaw}, we establish existence and uniqueness.  Furthermore, we prove how this principle can be used to prove how to prevent the congestion for specific choices of $V$ and $g.$
 
\medskip 
Before concluding this introduction, let us mention that in \cite{PQV1,Naomi}, the authors consider  a reaction term depending on the pressure variable \(p\), namely \(g = g(p)\), in the study of certain biological models. In these works, the parameter \(p\) is interpreted as the physical pressure. We believe that the results of the present paper could also be extended to this setting. Nevertheless, we shall not address these cases here, and instead restrict our attention to the situation where the reaction term depends primarily on the density \(u\). Indeed, the case \(g = g(\cdot,u)\) appears to be more suitable for the study of crowd motion phenomena, which constitutes one of the main motivations of this work.

\medskip   
\textbf{Plan of the paper :}
 Section $2$ outlines the assumptions and main results concerning existence, uniqueness, and the $L^1-$contraction principle.  We also discuss consequences of these results regarding congestion phenomena under specific assumptions on $V$ and $g.$ Section $3$ provides the proofs.  Building on results from \cite{Ig}, we first prove existence for locally Lipschitz continuous $g$ using Banach's fixed-point theorem.  Then, for general $g,$ we establish our main results using monotonicity arguments. In Section 4, we provide a detailed analysis of the example presented in \eqref{example1} and demonstrate how the main results of this paper can be applied to this specific case.
%
 
 \section{Main results}\label{sectionprem}
 
 We assume that   $\Omega\subset\RR^N$ is a bounded  open set, with regular boundary  of class $\C^2,$ splitted into regular partition $\partial \Omega=\Gamma_D\cup \Gamma_N,$ such that $\Gamma_D\cap \Gamma_N=\emptyset$ and 
 $$\mathcal L^{N-1}(\Gamma_D)> 0.$$
 We consider $H^1_0(\Omega)$ (resp. $H^1_D(\Omega)$) the usual  space of functions in the Sobolev space $H^1(\Omega),$  with  null trace on the boundary $\partial\Omega $ (resp. $\Gamma_D).$    
  
\medskip  
 We denote by $\signp$   the maximal monotone graph given by  
$$\signp(r)= \left\{ \begin{array}{ll}
	\displaystyle 1 &\hbox{ for  }r>0\\  
	\displaystyle [0,1]\quad & \hbox{ for } r=0\\
	\displaystyle 0 &\hbox{ for  }r<0. \end{array}\right. $$ 
Moreover, we define $\sign_0$ and  $\sign^{\pm}_0,$ the   discontinuous   applications defined from $\RR$ to $\RR$ by  
  $$\sign_0(r)= \left\{ \begin{array}{ll}
  	\displaystyle 1 &\hbox{ for  }r>0\\  
  	\displaystyle 0\quad & \hbox{ for } r=0\\
  	\displaystyle -1 &\hbox{ for  }r<0,  \end{array}\right.,\quad \spo(r)= \left\{ \begin{array}{ll}
  	\displaystyle 1 &\hbox{ for  }r>0\\   
  	\displaystyle 0 &\hbox{ for  }r\leq 0 \end{array}\right.    \quad   \smo(r)= \left\{ \begin{array}{ll}
  	\displaystyle 0 &\hbox{ for  }r\geq 0\\   
  	\displaystyle 1 &\hbox{ for  }r< 0.\end{array}\right.   $$

 \medskip 
Throughout the paper, we  assume that $u_0$ and the velocity vector filed $V$ satisfy  the following assumptions :  
 \begin{itemize}
 	\item  $\displaystyle  u _0\in L^\infty(\Omega)$  and $ \vert u_0\vert \leq 1$ a.e. in $\Omega.$
 	\item $\displaystyle V \in  W^{1,2}(\Omega)^N$, $\nabla \cdot V\in L^\infty(\Omega) $  and satisfies (an outward pointing  velocity vector field condition on the boundary)
 	\begin{equation} \label{HypV0}
 		V\cdot \nu \geq 0\quad \hbox{ on }\Gamma_D  \quad \hbox{ and } 	\quad V\cdot \nu = 0\quad \hbox{ on }\Gamma_N .
 		\end{equation}
 \end{itemize}  
For technical reason related to doubling and de-doubling  variables  techniques for the proof of uniqueness (cf. \cite{EIJ}),  we  assume  definitely that 
 \begin{equation}\label{HypVstg}
 	\liminf_{h\to 0}  \frac{1}{h}  \int_{\Omega\setminus \Omega_h}  \xi\:   V(x)   \cdot \nu(\pi(x))  \: dx \geq   0 , \quad \hbox{ for any }0\leq \xi \in L^2(\Omega).
 \end{equation}
 
 
\medskip 
\begin{definition}\label{def0}  For a given $f\in L^1(Q)$, a   couple $(u ,p) $  is said to be a weak solution of  the problem 
	\begin{equation}
		\label{cmef}
		\left\{  \begin{array}{ll}\left.
			\begin{array}{l}
				\displaystyle \frac{\partial u }{\partial t}  -\Delta p +\nabla \cdot (u  \: V)= f \\
				\displaystyle u \in \sign(p)\end{array}\right\}
			\quad  & \hbox{ in } Q \\  \\
			\displaystyle p= 0  & \hbox{ on }\Sigma_D \\  \\
			\displaystyle (\nabla p- u  \: V)\cdot \nu = 0  & \hbox{ on }\Sigma_N\\  \\
			\displaystyle  u (0)=u _0 &\hbox{ in }\Omega,\end{array} \right.
	\end{equation}  
	if $(u ,p)\in  L^\infty(Q) \times  L^2
	\left(0,T;H^1_D(\Omega)\right)$, $\displaystyle    u \in \sign (p)$ a.e. in $Q$  and
	\begin{equation}
		\label{evolwf}
		\displaystyle \frac{d}{dt}\int_\Omega u \:\xi+\int_\Omega ( \nabla p -  u \:V) \cdot  \nabla\xi   =     \int_\Omega f\: \xi  , \quad \hbox{ in }{\D}'(0,T),\quad \forall \: 	\xi\in H^1_D(\Omega).
	\end{equation}
	
	We'll say plainly  that $u$ is a solution of \eqref{cmef} if $u\in \C([0,T),L^1(\Omega))$, $u(0)=u_0$ and there exists $p\in L^2
	\left(0,T;H^1_D(\Omega)\right)$ such that the couple $(u,p)$ is a weak solution of   \eqref{cmef}.
\end{definition} 

\medskip  
 Now, consider the reaction-diffusion problem 
 \begin{equation}\label{cmeg}
 	\left\{  \begin{array}{ll}\left.
 		\begin{array}{l}
 			\displaystyle \frac{\partial u }{\partial t}  -\Delta p +\nabla \cdot (u  \: V)= g(.,u) \\
 			\displaystyle u \in \sign(p)\end{array}\right\}
 		\quad  & \hbox{ in } Q \\    \\ 
 		\displaystyle p= 0  & \hbox{ on }\Sigma_D \\    \\ 
 		\displaystyle (\nabla p- u  \: V)\cdot \nu = 0  & \hbox{ on }\Sigma_N \\  \\ 
 		\displaystyle  u (0)=u _0 &\hbox{ in }\Omega,\end{array} \right.
 \end{equation}
 where   $g$ : $Q\times\RR\rightarrow\RR$ is a Caratheodory function ; i.e.  continuous in $r\in\RR$ and measurable in  $(t,x)\in Q$. We assume that $g$ satisfies moreover, the following assumptions :

 \begin{itemize}
 	\item [($\G_1)$]  $g^+(.,-1)\in L^2(Q) $  and $ g^-(.,1)\in L^2(Q)$ 
 	\item [($\G_2)$]   There exists $0\leq \R\in L^\infty(0,T),$ s.t. for any $a,\: b\: \in [-1,+1],$   	we have 
 	$$\spo(b-a) \: (g(t,x,b)-g(t,x,a)) \leq \R(t)  \: (b-a)^+,\quad  \hbox{ for a.e. }(t,x)\in Q.$$
 \end{itemize}
 On sees in particular that  $(\G_2)$ implies   
 \begin{equation}\label{explicitcondg} 
 	-g^-(.,1) - \R \: (1-r)\leq g(.,r)\leq g^+(.,-1) + \R \: (1+r),\quad \hbox{ for any }r\in [-1,+1],
 \end{equation}
 so that 
 $g(.,r)\in L^2(Q) ,$   for any $r\in [-1,+1].$  
  
 \begin{remark}
 	See that the condition ($\G_2)$) is fulfilled for instance whenever 
 	\begin{equation} 
 		\frac{\partial g}{\partial r}(t,x,.)\leq  \theta \quad \hbox{ in } \D'([-1,1]),\hbox{   for a.e. }(t,x)\in Q\hbox{ with } 0\leq \theta  \in \C([-1,1] ),
 	\end{equation}
 	In this case $\R= \max_{-1\leq r\leq 1} \theta(r).$    Here one can take $\theta$ depending on $t$ as well. 
 \end{remark}
 
 \begin{theorem}\label{texistg}
 	Under the assumption $(\G_1)$ and $(\G_2),$ for any $u_0\in L^\infty(\Omega)$    such that $\vert u_0\vert \leq 1,$ a.e. in $\Omega,$ the problem \eqref{cmeg} has a unique   solution $u$, in the sense of Definition \ref{def0} with \begin{equation}\label{fg}
 	 f(t,x)=g(t,x,u(t,x))  \quad \hbox{ for a.e. } (t,x)\in Q.  
 	\end{equation}
 	Moreover,  if $u_1$ and $u_2$ are two  solutions of \eqref{cmeg} associated with $g_1$ and $ g_2,$ respectively, then there exists $\kappa \in L^\infty(Q),$ such that $\kappa\in \signp(u_1-u_2)$ a.e. in $Q$ and 
 		\begin{equation}
 			\label{evolineqcomgp}
 			\frac{d}{dt}	\int_\Omega ( u _1-u _2)^+ \: dx \leq \int_\Omega \kappa\:     ( g_1(.,u_1)-g_2(.,u_2))\: dx ,\quad \hbox{ in }\D'(0,T).
 		\end{equation}
 		In particular, there exists $\tilde \kappa\in \sign(u_1-u_2)$ a.e. in $Q$, such that 
 		\begin{equation}	\label{evolineqcontg} 
 			\frac{d}{dt} \Vert u_1-u_2\Vert_{1} \leq \int ( g_1(.,u_1)-g_2(.,u_2)) \tilde \kappa \: dx,\quad \hbox{ in }\D'(0,T).		\end{equation} 
  
 \end{theorem} 
 
\begin{remark}\label{Cuniq}
Let  $u_0\in L^\infty(\Omega)$ and $f\in L^1(Q).$   As a consequence of Theorem \ref{texistg},  there exists a unique  $u\in L^\infty(Q)$ such that, there exists $p\in L^2
\left(0,T;H^1_D(\Omega)\right)$, $\displaystyle    u \in \sign (p)$ a.e. in $Q$  and the couple $(u ,p)$ satisfies 
\begin{equation}
	\displaystyle   - \int_Qu\: \xi  \:\partial_t \psi \: dxdt +\int_Q ( \nabla p -  u \:V) \cdot  \nabla\xi \: \psi\: dxdt   =     \int_Q g(.,u)\: \xi \psi\: dxdt + \int_\Omega u_0\: \psi(0)\:\xi \: dx  , 
\end{equation} 
for any $	\psi\in \D([0,T))\hbox{ and }\xi\in H^1_D(\Omega).$
   
\end{remark}

 \medskip
 \begin{theorem} 
 	Under the assumption ($\G_1)$ and ($\G_2)$, for any for any $u_0\in L^\infty(\Omega)$    such that $\vert u_0\vert \leq 1,$ a.e. in $\Omega,$  and $\omega_1,\: \omega_2  \in W^{1,1}(0,T)$ satisfying   $ u_0\leq \omega_2(0)$ (resp. $\omega_1(0)\leq u_0 $) and, for any $t\in (0,T),$  
 	\begin{equation}\label{sup}
 		\dot \omega_2(t)+ \omega_2(t)\nabla \cdot V \geq  g(.,\omega_2(t))\quad \hbox{ a.e. in  }\Omega
 	\end{equation}
 	\begin{equation}\label{sub}
 		\hbox{(resp. } \dot \omega_1(t)+ \omega_1(t)\nabla \cdot V \leq   g(.,\omega_1(t))\quad \hbox{ a.e. in  }\Omega), 
 	\end{equation}   we have  
 	\begin{equation}\label{compw}
 		u\leq \min(1,\omega_2) \quad (\hbox{resp. }\max(-1,\omega_1)\leq u)\quad \hbox{  a.e. in }Q. 
 	\end{equation}

 \end{theorem}

 \begin{theorem}\label{texistg1}
 	Under the assumptions of Theorem \ref{texistg}, let us consider $(u,p)$ the solution   of the problem \eqref{cmeg} and $w_i$  as given by \eqref{sup}  and \eqref{sub}, for $i=1,2.$  Assume that $-1\leq w_1(0)\leq w_2(0)\leq 1,$  and define $0\leq \tau_c\leq T$ by 
 	\begin{equation}
 		\tau_c=\max\{ \tau\in [0,T)\: :\: 	-1\leq w_1(t)\leq w_2(t)\leq 1 \quad \hbox{ for any }t\in [0,\tau[\} . 
 	\end{equation}
 	Then, the solution   of    \eqref{cmeg}  is the unique weak solution (in $[0,\tau_c)$)  of the   constrained reaction-transport equation 
 	\begin{equation}
 		\label{cmep0}
 		\left\{  \begin{array}{ll} 
 				\displaystyle \frac{\partial u }{\partial t}   +\nabla \cdot (u  \: V)= g(t,x,u), \quad 
 				\vert u\vert \leq 1  	\quad  & \hbox{ in }   (0,\tau_c)\times \Omega \\   \\  
 			\displaystyle  u (0)=u _0 &\hbox{ in }\Omega,\end{array} \right.
 	\end{equation}
 	in the sense that, $u  \in \C([0,\tau_c),L^1(\Omega)),$  $\vert u\vert \leq 1,$    a.e. in $(0,\tau_c)\times \Omega,$  $u(0)=u_0$   and
 	\begin{equation}
 		\label{evolwg}
 		\displaystyle \frac{d}{dt}\int_\Omega u \:\xi- \int_\Omega   u \:V \cdot  \nabla\xi   =     \int_\Omega  g(.,u)\: \xi  , \quad \hbox{ in }{\D}'(0,\tau_c), \quad \forall \: \xi\in H^1_D(\Omega).
 	\end{equation}
 \end{theorem}

 The following corollaries  show some particular cases which could be of particular interest for applications.  
 
 \begin{corollary} \label{cexistg1}  
 	Under the assumption of Theorem \ref{texistg}, if $g$ satisfies moreover 
 	\begin{equation}
 		( \G_3) \quad  \nabla \cdot V \geq g(., 1),  \hbox{ a.e.  in } Q   \quad \hbox{ and } \quad  
 		(\G_4)   \quad\nabla \cdot V \leq g(.,-1),  \hbox{ a.e.  in }  Q,  
 	\end{equation}
 	then the solution of  \eqref{cmeg}  is the unique solution of  \eqref{cmep0} in  $[0,T)$. 
 \end{corollary} 
 \begin{proof}
 	It is enough to take $\omega_2(t)=-\omega_1(t)= 1,$ for any $t\in [0,T),$ and apply Theorem \ref{texistg1}  with $\tau_c=T.$  
 \end{proof}
 
 \begin{remark}
 	See that     $(\G_3)$ and $(\G_4) $   constitute sufficient conditions to lessen the so called congestion phenomena in the transport process. This is an interesting property which could be used eventually for application in biological models.   Indeed, heuristically  the reaction term $g(t,x,1)$  proceeds like a source or an absorption   at the position $x\in \Omega$ and time $t,$ in a ''congested''  circumstance. In the case where, performing  its value  regarding to the divergence of the velocity vector field   $V$ may avoid the congestion in the transportation phenomena.

 \end{remark}

 \begin{corollary} \label{cexistg} Under the assumptions of Theorem \ref{texistg}, assume moreover that $0\leq u_0\leq 1$ a.e. in $ \Omega$ and 
 	\begin{equation}
 		( \G_5) \quad  0 \leq g(.,0)  \hbox{  a.e. in }
 		Q .\end{equation} 
 	then    
 	\begin{equation}\label{solpos}
 		p\geq 0 \hbox{ and } 	0\leq u\leq 1,\quad \hbox{  a.e. in }Q.
 	\end{equation}   
 	Moreover,  if $g$ satisfies $(\G_3),$ then the solution of  \eqref{cmeg}  is the unique solution of  
 	\begin{equation} 
 		\left\{  \begin{array}{ll}
 			\left. \begin{array}{l}
 				\displaystyle \frac{\partial u }{\partial t}   +\nabla \cdot (u  \: V)= g(t,x,u)\\ 
 				0\leq  u  \leq 1 
 			\end{array}\right\} 	\quad  & \hbox{ in }   (0,T)\times \Omega \\  
 			\displaystyle u  \: V\cdot \nu = 0  & \hbox{ on }(0,T)\times \Gamma_N  \\   
 			\displaystyle  u (0)=u _0 &\hbox{ in }\Omega.\end{array} \right.
 	\end{equation}

 \end{corollary}
 \begin{proof}
 	It is enough to take $w_1(t)=0$ and $\omega_2(t)= 1,$ for any $t\in [0,T),$ and apply Theorem \ref{texistg1}  with $\tau_c=T.$  
 \end{proof}

 \section{Proofs  }\label{sectionRD}
 \setcounter{equation}{0}

\begin{theorem}\label{texistf}(cf. \cite{Ig})
	For any $f\in L^2(Q)$ and $u_0\in L^\infty(\Omega)$ be such that $\vert u_0\vert \leq 1$ a.e. in $\Omega,$ the problem \eqref{cmef} has a unique   solution  $u\in \C([0,T),L^1(\Omega))$ and $u(0)=u_0$ in the sense of Definition \ref{def0}. Moreover,  we have  
	
	\begin{enumerate}
		
		\item If $(u_1,p_1)$ and $(u_2,p_2)$ are two weak solutions of \eqref{cmef} associated with $f_1,\ f_2\: \in L^1(Q)$ respectively, then there exists $\kappa \in L^\infty(Q),$ such that $\kappa\in \signp(u_1-u_2)$ a.e. in $Q$ and 
		\begin{equation}
			\label{evolineqcomp}
			\frac{d}{dt}	\int_\Omega ( u _1-u _2)^+ \: dx \leq \int_\Omega \kappa\:     ( f_1-f_2)\: dx ,\quad \hbox{ in }\D'(0,T).
		\end{equation}
		In particular, there exists $\tilde \kappa\in \sign(u_1-u_2)$ a.e. in $Q$, such that 
		\begin{equation}	\label{evolineqcont} 
			\frac{d}{dt} \Vert u_1-u_2\Vert_{1} \leq \int ( f_1-f_2)\tilde \kappa \: dx,\quad \hbox{ in }\D'(0,T).		\end{equation} 
		Moreover, if $f_1\leq f_2,$  a.e. in  $Q,$ and $u_1,u_2$ are two corresponding solutions  satisfying $u_1(0)\leq u_2(0) $ a.e. in $\Omega,$   then
		$$u_1\leq u_2,\quad \hbox{ a.e. in  }Q.$$

		\item  If  $u_0\geq 0$ and $f\geq 0,$ then $(u,p)$ is the unique  weak solution of the one phase problem
		\begin{equation} 			\left\{  \begin{array}{ll}\left.
				\begin{array}{l}
					\displaystyle \frac{\partial u }{\partial t}  -\Delta p +\nabla \cdot (u  \: V)= f \\
					\displaystyle  0\leq u\leq 1,\:  p\geq 0,\:  p(u-1)=0 \end{array}\right\}
				\quad  & \hbox{ in } Q \\  
				\displaystyle p= 0  & \hbox{ on }\Sigma_D \\  
				\displaystyle (\nabla p- u  \: V)\cdot \nu = 0  & \hbox{ on }\Sigma_N  \\ 
				\displaystyle  u (0)=u _0 &\hbox{ in }\Omega.\end{array} \right. 
		\end{equation}

		\item\label{contdep} For each $n=1,2,....,$ let  $f_n\in L^1(Q),$   $u_{0n}\in L^\infty(\Omega)$ be such that $\vert u_{0n}\vert \leq 1$ a.e. in $\Omega,$ and $u_n$   be  the unique weak solution of the corresponding problem \eqref{cmef}. If $u_{0n}\to u_0$ in $L^1(\Omega)$ and $f_n\to f$ in $L^1(Q),$ then $	u_n \to u,$  in $ \C([0,T),L^1(\Omega))$,    $	p_n \to p, $ in $ L^2(0,T;H^1_D(\Omega))-\hbox{weak},$ and $(u,p)$ is the solution corresponding to $u_0$ and $f.$

	\end{enumerate} 
\end{theorem}

\medskip
  \begin{proposition}\label{ptexistg}
Under the assumption $(\G_1),$ assume that  $g(.,r)$ is a  Lipschitz continuous function with respect to $r\in [-1,1]$  ; i.e. there exists $L_g>0,$ such that 
\begin{equation}
	\vert g(.,r_1)-g(.,r_2)\vert \leq L_g\: \vert r_1-r_2\vert,\quad \hbox{ a.e. in }Q,\hbox{ for any } r_1,r_2\in [-1,1].
\end{equation}
Then,  for any $u_0\in L^\infty(\Omega)$    such that $\vert u_0\vert \leq 1,$ a.e. in $\Omega,$ the problem \eqref{cmeg} has a  weak solution $(u,p)$. 
 \end{proposition} 

 \begin{proof}[Proof]  To prove  with Banach fixed point arguments. Indeed,   thanks to Theorem \ref{texistf}, let consider the sequence $(u_n,p_n)_{n\in \NN}$ given by the solution of the problem 
	\begin{equation}
		\left\{  \begin{array}{ll}\left.
			\begin{array}{l}
				\displaystyle \frac{\partial u_{n} }{\partial t}  -\Delta p_{n}  +\nabla \cdot (u_{n}   \: V)= g(., u_{n-1})  \\
				\displaystyle u_{n}  \in \sign(p_{n} )\end{array}\right\}
			\quad  & \hbox{ in } Q \\   
			\displaystyle p_{n}= 0  & \hbox{ on }\Sigma_D \\  
			\displaystyle (\nabla p_{n}- u_{n}  \: V)\cdot \nu = 0  & \hbox{ on }\Sigma_N \\  
			\displaystyle  u_{n} (0)=u _0 &\hbox{ in }\Omega.\end{array} \right.
	\end{equation}
	  Using \eqref{evolineqcont} and the fact that $g$ is a  Lipschitz , for any $n=1,2,...,$ we have 
 $$\Vert u_{n+1}(t)-u_{n}(t)\Vert_1 \leq   L_g\:  \int_0^t \Vert u_{n}-u_{n-1}\Vert_1,\quad \hbox{ for any }t\in [0,T),$$ 
 	so that, by  iterating, we have  
 	$$\Vert u_{n+1}(t)-u_{n}(t)\Vert_1 \leq  \frac{( L_g\: T)^n}{(n-1)! }  \Vert u_{1}-u_{0}\Vert_{L^1(Q)} ,\quad \hbox{ for any }t\in [0,T).$$  
 	This implies that there exists $u\in \C([0,T),L^1(\Omega))$ such that $u(0)=u_0$ and 
	 	\begin{equation}\label{convun}
 	u_n\to u,\quad \hbox{ in } \C([0,T),L^1(\Omega)), \hbox{ as }n\to\infty.
	 	\end{equation}
 	Now, let us prove that $p_n \to p$ is bounded and converge weakly in $L^2(0,T;H^1_D(\Omega))$. Thanks to Lemma $2.1$ of \cite{Ig}, we have
	$$ -\Delta p^+_{n}  +\big(\nabla \cdot (u_{n}   \: V)-g(., u_{n-1}) \big) \sign_0^+(p) \leq 0\qquad \mbox{ in }\;\;  \mathcal{D}^\prime\big( (0,T)\times \overline{\Omega}\big) $$
	and 
		$$ -\Delta p^-_{n}  +\big(\nabla \cdot (u_{n}   \: V)+g(., u_{n-1}) \big) \sign_0^-(p) \leq 0\qquad \mbox{ in }\;\;  \mathcal{D}^\prime\big( (0,T)\times \overline{\Omega}\big).$$
\noindent Taking $p^+_n$  (res.  $p^-_n$ )as a test function and integration over $(0,\tau)\times \Omega$, we obtain 
	
	\begin{eqnarray*}
	\displaystyle  \int_0^\tau\int_\Omega  |\nabla p^+_n|^2 dx  \leq  \int_0^\tau\int_\Omega u_n \:V \cdot \nabla p^+_n  dx +    \int_0^\tau\int_\Omega g(.,u_{n-1}) p^+_n dx  \quad  \mbox{and}  \\ \\
	\displaystyle    \int_0^\tau\int_\Omega  |\nabla p^-_n|^2 dx  \leq  \int_0^\tau\int_\Omega u_n \:V\cdot \nabla p^-_n  dx -    \int_0^\tau\int_\Omega g(.,u_{n-1}) p^-_n dx.
			\end{eqnarray*} 
			
\noindent Using Young inequality, we have
\begin{eqnarray*}
		  \int_0^\tau\int_\Omega u_n \:V \cdot\nabla p^\pm_n  dx   &\leq & \frac{3}{4} \int_0^\tau\int_\Omega  \: |\: V\:  |^2  dx +\frac{1}{3} \int_0^\tau\int_\Omega  \: |\:  \nabla p^\pm_n \:  |^2  dx.   
	\end{eqnarray*} 
Since  $g(.,r)\in L^2(Q)$  for any $r\in [-1,+1],$  using Pincaré inequality after Young inequality, we obtain
\begin{eqnarray*}
		\int_0^\tau\int_\Omega g(.,u_{n-1}) p^\pm_n dx   &\leq & C \int_0^\tau\int_\Omega  \: |\: g(.,u_{n-1})\:  |^2  dx +\frac{1}{3} \int_0^\tau\int_\Omega  \: |\:  \nabla p^\pm_n \:  |^2  dx.   
	\end{eqnarray*} 
where $C$ is independent constant of $n$.  Consequently

\begin{equation}\label{evolwf2}
		\displaystyle   \frac{1}{3}  \int_0^\tau \int_\Omega  |\nabla p^\pm_n|^2 dx  \leq   \frac{3}{4} \int_0^\tau\int_\Omega  \: |\: V\:  |^2  dx + \: C(N,\Omega) \int_0^\tau\int_\Omega  \: |\: g(.,u_{n-1})\:  |^2  dx .
			\end{equation}
 Unsing \eqref{explicitcondg}, we see that the sequence \( p_n \)  is   bounded in \( L^2(0,T;H^1_D(\Omega)) \). Therefore, there exists a subsequence, which we denote again by 
 \( p_n \)  such that 
 $$p_n \to p \quad \mbox{ weakly in} \; \;  L^2(0,T;H^1_D(\Omega)) .$$ 
 Letting $ n\to \infty $ and using \eqref{convun} with   monotonicity argument, 
we conclude that \( (u, p) \) forms a weak solution to (\ref{cmef}).
  \end{proof}
 
\medskip
\begin{proof}[Proof of Theorem \ref{texistg}] \underline{Existence:}  To prove existence, we  consider   $(u_{\lambda,\mu}, p_{\lambda,\mu})$  the solution of the  following approximated problem   
	\begin{equation}
		\left\{  \begin{array}{ll}\left.
			\begin{array}{l}
				\displaystyle \frac{\partial u }{\partial t}  -\Delta p  +\nabla \cdot (u  \: V)= \tilde{g}_{\lambda,\mu}(., u) +\R(t) \: u   \\
				\displaystyle u  \in \sign(p)\end{array}\right\}
			\quad  & \hbox{ in } Q \\   
			\displaystyle p= 0  & \hbox{ on }\Sigma_D \\  
			\displaystyle (\nabla p- u \: V)\cdot \nu = 0  & \hbox{ on }\Sigma_N \\  
			\displaystyle  u(0)=u _0 &\hbox{ in }\Omega,\end{array} \right.
	\end{equation}
where  $\displaystyle  \tilde{g}(., r) =   g(., r) -\R(t)\:  r, \;  \tilde{g}_{\lambda,\nu}(., r) = \tilde{g}_{\lambda}(., r^{+}) + \tilde{g}_{\mu}(., -r^{-}) $ and  for a.e. $\displaystyle (t,x)\in Q$ the function   $\tilde{g}_{\delta }(x, .)$ is the  $\delta$-Yoshida approximation of $\displaystyle g(x, .)$  given by  
$$\displaystyle \tilde{g}_\delta(t,x,.)= \delta\Big(I-\big(I+\frac{1}{\delta} \tilde{g}(t,x,.)\big)^{-1} \Big) \quad \mbox{for } \;\; \delta \in{\{\lambda, \: \mu\}}. $$
It is not difficult to see that the function  $\displaystyle  r \to  \tilde{g}(., r)$ is a nondecreasing function. Now, let  $\lambda> \tilde{\lambda}>0$ and $\nu>0$ using \eqref{evolineqcont},  , for   $\displaystyle  \kappa\in \signp \big(u_{\lambda,\nu}(t)-u_{\tilde{\lambda},\nu}(t) \big)$, we have 
\begin{equation}\label{esimtExistence0}
 \displaystyle \frac{d}{dt}	\int_\Omega  \big(u_{\lambda,\nu}(t)-u_{\tilde{\lambda},\nu}(t) \big)^{+}  dx \leq   \int_\Omega \kappa  \big( \tilde{g}_{\lambda,\mu}(., u_{\lambda,\mu}) - \tilde{g}_{\tilde{\lambda},\mu}(., u_{\tilde{\lambda},\mu} ) \big)\: dx + \R(t)\:   \int_\Omega   \big(u_{\lambda,\nu}(t)-u_{\tilde{\lambda},\nu}(t) \big)^{+}  dx.
	\end{equation}
We have 
$$
\begin{array}{ll}
  I: &\displaystyle  = \int_\Omega \kappa  \big( \tilde{g}_{\lambda,\mu}(., u_{\lambda,\mu}) - \tilde{g}_{\tilde{\lambda},\mu}(., u_{\tilde{\lambda},\mu} ) \big)\: dx  \\
 \;& \displaystyle  =   \int_\Omega \kappa  \big(  \tilde{g}_{\lambda}(., u_{\lambda,\nu}^{+}) - \tilde{g}_{\tilde{\lambda}}(., u_{\tilde{\lambda},\nu}^{+})\big)\: dx +  \int_\Omega \kappa  \big(  \tilde{g}_{\mu}(., -u_{\lambda,\nu}^{-}) - \tilde{g}_{\mu}(., -u_{\tilde{\lambda},\nu}^{-})\big)\: dx, \\
 \displaystyle & \displaystyle \leq   \int_{[u_{\lambda,\nu} \geq u_{\tilde{\lambda},\nu} ]} \big(  \tilde{g}_{\lambda}(., u_{\lambda,\nu}^{+}) - \tilde{g}_{\tilde{\lambda}}(., u_{\tilde{\lambda},\nu}^{+})\big)\: dx +  \int_{[u_{\lambda,\nu} \geq u_{\tilde{\lambda},\nu} ]}   \big(  \tilde{g}_{\mu}(., -u_{\lambda,\nu}^{-}) - \tilde{g}_{\mu}(., -u_{\tilde{\lambda},\nu}^{-})\big)\: dx, \\
 \displaystyle & \displaystyle \leq   \int_{[u_{\lambda,\nu} \geq u_{\tilde{\lambda},\nu} ]} \big(  \tilde{g}_{\lambda}(., u_{\lambda,\nu}^{+}) - \tilde{g}_{\lambda}(., u_{\tilde{\lambda},\nu}^{+})\big)\: dx + 
  \displaystyle \   \int_{[u_{\lambda,\nu} \geq u_{\tilde{\lambda},\nu} ]} \big(  \tilde{g}_{\lambda}(., u_{\tilde{\lambda},\nu}^{+}) - \tilde{g}_{\tilde{\lambda}}(., u_{\tilde{\lambda},\nu}^{+})\big)\: dx \\ 
 \displaystyle &  \displaystyle+ \int_{[u_{\lambda,\nu} \geq u_{\tilde{\lambda},\nu} ]}   \big(  \tilde{g}_{\mu}(., -u_{\lambda,\nu}^{-}) - \tilde{g}_{\mu}(., -u_{\tilde{\lambda},\nu}^{-})\big)\: dx. \\
 \end{array} 
$$
Using the monotonicity  of the functions  $\tilde{g}_{\lambda}$, $\tilde{g}_{\mu}$ and the fact that $ \displaystyle \tilde{g}_{\lambda}(.,r) \leq \tilde{g}_{\tilde{\lambda}}(., r)$ for $\lambda \geq \tilde{\lambda}$ and $r\geq 0$, we observe that  $$ \int_\Omega \kappa  \big( \tilde{g}_{\lambda,\mu}(., u_{\lambda,\mu}) - \tilde{g}_{\tilde{\lambda},\mu}(., u_{\tilde{\lambda},\mu} ) \big)\: dx \leq 0.$$
  From (\ref{esimtExistence0}), we have 
  $$ \frac{d}{dt}	\int_\Omega  \big(u_{\lambda,\nu}(t)-u_{\tilde{\lambda},\nu}(t) \big)^{+}  dx \leq \R(t)\:  \int_\Omega  \big(u_{\lambda,\nu}(t)-u_{\tilde{\lambda},\nu}(t) \big)^{+}  dx.$$
Using the Gronwall  inequality, we obtain that
       $$  u_{\lambda,\mu} \leq  u_{\tilde{\lambda},\mu}   \quad \mbox{a.e. in } Q.$$
Similarly, it can be shown that for any \(\nu > \tilde{\nu} > 0\) and any \(\lambda > 0\), we have \(u_{\lambda,\nu} \geq u_{\lambda,\tilde{\nu}}\) almost everywhere on \(Q\). The fact that $|u_{\lambda,\nu} |\leq 1$, combined with the monotone convergence theorem, implies that:  
\[
u_{\lambda,\nu} \xrightarrow[\lambda]{\downarrow} u_{0,\nu} \xrightarrow[\nu]{\uparrow} u, \quad \text{and} \quad u_{\lambda,\nu} \xrightarrow[\nu]{\uparrow} u_{\lambda,0} \xrightarrow[\lambda]{\downarrow} u \quad \text{in } L^1(Q).
\]
Using a diagonal process, it is possible to find a function \(\nu(\lambda)\) such that \(u_\lambda := u_{\lambda,\nu(\lambda)} \to u\) in \(L^1(Q)\). If necessary, we may extract a subsequence to achieve this.
  Similarly as in the proof of Proposition \ref{ptexistg}, one can prove that 
$$p_{\lambda,\nu(\lambda)} \to p \qquad \quad \mbox{ weakly in} \; \;  L^2(0,T;H^1_D(\Omega))\quad \mbox{as} \; {\lambda} \to 0.$$
Letting ${\lambda} \to 0 $, in the weak formulation,  we deduce the existence  result of solution.  \\

\medskip 
\noindent \underline{Uniqueness:} 	
	The uniqueness is a simple consequence of the contraction property and  Growall Lemma. Indeed, if $u_1$ and $u_2$ are two solutions then, by using \eqref{evolineqcomp},  we have 
	\begin{eqnarray*}
		\frac{d}{dt} \int (u_1-u_2)^+ &\leq & \int_{[u_1\geq u_2 ]} (g(.,u_1)-  g(.,u_2)  ) \\ 
		&\leq&      \int_{[u_1\geq u_2 ]}   ( g(.,u_1) -g(.,u_2)    )^+  \\
		&\leq&    \R \int (u_1-u_2)^+.
	\end{eqnarray*} 
	Using Gronwall, this implies that $u_1\leq  u_2.$ In the same way, we can prove that $u_2\leq u_1.$ Thus the uniqueness. 
\end{proof}

The estimates  \eqref{compw}  could be   interesting for the control of the congestion. Indeed, since $p$ is concentrated in the congestion zone $[\vert u\vert =1],$ one sees that   performing the values of $w_1$ and $w_2,$ we may control/avoid the emergence of these zones.  The following result give some information about the time of the emergence of congestion zone whenever  $[\vert u\vert <1].$

 \begin{proof}[Proof of Theorem \ref{texistg1}] 
 	It is clear that any solution of  \eqref{cmep0} in  $[0,\tau)$, for a given $0<\tau\leq T,$   is a solution of \eqref{cmeg} in  $[0,\tau)$.  Thanks to the uniqueness the converse part  happens only  if    \eqref{cmep0} has a solution  in  $[0,\tau)$. Let us prove that this true in $[0,\tau_c). $ 
 	To this aim, we proceed by using mainly  existence and comparison principle results of Theorem \ref{texistg}.  We fix an arbitrary  $  \alpha>0 $ and,  we consider the reaction term $g_\alpha$   given by
 	$$g_\alpha (.,r)= \alpha \: g(.,r/\alpha),\quad \hbox{ a.e. in }Q,\: \quad \hbox{ for any }r\in \RR.$$
 	It is clear that $g_\alpha$ satisfies all the assumptions of Theorem \ref{texistg}. So, we can consider  $(v_\alpha,p_\alpha)$ the    solution of
 	\begin{equation}
 		\label{cmea}
 		\left\{  \begin{array}{ll}\left.
 			\begin{array}{l}
 				\displaystyle \frac{\partial   v_\alpha }{\partial t}  -\Delta    p_\alpha +\nabla \cdot (  v_\alpha  \: V)= g_\alpha (.,v_\alpha)\\
 				\displaystyle    v_\alpha \in \sign(   p_\alpha)\end{array}\right\}
 			\quad  & \hbox{ in } Q \\   
 			\displaystyle   p_\alpha= 0  & \hbox{ on }\Sigma_D \\   
 			\displaystyle (\nabla   p_\alpha-    v_\alpha  \: V)\cdot \nu = 0  & \hbox{ on }\Sigma_N  \\   
 			v_\alpha(0)=\alpha \: u_0 & \hbox{ in  }\Omega .\end{array} \right.
 	\end{equation}
 	Our aim is to prove that  $u_\alpha :=v_\alpha/\alpha$ is in fact a solution of   \eqref{cmep0} in  $[0,\tau)$  in $[0,\tau_c)$ as claimed.  Thanks to \eqref{sup}, we have 
 	$$ \partial_t(\alpha\omega_2) + \alpha\omega_2  \nabla \cdot V \geq g_\alpha (., \alpha\omega_2),  \quad  \hbox{ in }Q,  $$
 	and  $v_\alpha (0)\leq \alpha \omega_2(0).$  Thanks to Theorem \ref{texistg}, we have  $v_\alpha \leq \alpha\omega_2$ a.e. in $Q.$   In a similar way,  we can prove   using \eqref{sub}, that    $v_\alpha \geq  -\alpha\omega_1$ a.e. in $Q.$  So,  under the assumptions  \eqref{sup}-\eqref{sub}, we have 
 	$$\vert v_\alpha (t,x)\vert \leq \alpha\max(\vert w_1(t)\vert ,\vert w_2(t)\vert ),\quad \hbox{ a.e.   } (t,x)\in Q.$$    
 	Remember that, by definition of $\tau_c,$   $m_c:=\max_{t\in [0,\tau_c]} \max(\vert w_1(t)\vert ,\vert w_2(t)\vert )\leq 1,$ so that   $\vert v_\alpha\vert \leq \alpha <1,$  a.e. in $Q,$  and then   $p_\alpha \equiv 0.$  This implies that   $v_\alpha$ is in fact a weak solution of
 	\begin{equation}
 		\left\{  \begin{array}{ll}
 			\left. \begin{array}{l}
 				\displaystyle \frac{\partial v_\alpha }{\partial t}   +\nabla \cdot (v_\alpha  \: V)= g_\alpha(t,x, v_\alpha)\\  
 				\vert v_\alpha \vert \leq \alpha \end{array}\right\} 
 			\quad  & \hbox{ in } (0,\tau_c)\times \Omega  \\  \\
 			\displaystyle v_\alpha  \: V\cdot \nu = 0  & \hbox{ on } (0,\tau_c)\times \Gamma_N\\   \\
 			v_\alpha(0)=\alpha \: u_0 & \hbox{ in  }\Omega 	. \end{array} \right.
 	\end{equation}
 	Now,   taking  $u_\alpha =v_\alpha /\alpha,$ we see that $\vert u_\alpha \vert \leq 1 $  and $u_\alpha$ is a weak  solution of 
 	\begin{equation}
 		\left\{  \begin{array}{ll}
 			\left. \begin{array}{l}
 				\displaystyle \frac{\partial u_\alpha }{\partial t}   +\nabla \cdot (u_\alpha  \: V)= g(t,x, u_\alpha)\\  
 				\vert u_\alpha \vert \leq 1  \end{array}\right\} 
 			\quad  & \hbox{ in } (0,\tau_c)\times \Omega  \\  \\
 			\displaystyle u_\alpha  \: V\cdot \nu = 0  & \hbox{ on }(0,\tau_c)\times \Gamma_N \\   \\
 			u_\alpha(0)=   u_0 & \hbox{ in  }\Omega .\end{array} \right.
 	\end{equation} 
 	Thus the result of the theorem. 
 \end{proof}

 \section{Example }
 
 A typical example maybe given by the absorption term 
\begin{equation}\label{absorption0}
	 g= -\alpha u ( u -  u_{\text{eq}}) 
\end{equation}
to  represent a process where the quantity $u$ of active population is being "absorbed" or removed from the system at a rate proportional to its difference from a target value $u_{\text{eq}}$. This type of reaction term is commonly used in   models of population dynamics, where each component works as follows  
 
 \begin{itemize}

 	\item $0<\alpha $ is a positive constant called the \textit{absorption rate} or \textit{absorption coefficient}. It determines how quickly $u$ is absorbed. A larger $\alpha$ means faster absorption.
 	
 	\item $0\leq u_{\text{eq}}\leq 1$  is the \textit{target value} or \textit{equilibrium value}. Formally, the absorption term pulls $u$ towards $u_{\text{eq}}$.
 	
 	\item $(u - u_{\text{eq}})$ represents how far $u$ is from its target value.
 \end{itemize}
The  negative sign,   $-$,  indicates that the reaction \textit{reduces} the value of $u$ if $u > u_{\text{eq}}$:  the further $u$ is above $u_{\text{eq}}$, the faster it decreases.  It indicates also that the reaction \textit{increase} the value of $u$ if $u < u_{\text{eq}}:$ the further $u$ is below $ u_{\text{eq}}$, the faster it increases (towards $ u_{\text{eq}}$).  

\bigskip 

Consider an active  population with density $u = u(x,t)$ moving along trajectories defined by the velocity field $V$ within a given environment $\Omega$. Assume that the movement is related to an optimal mobility capacity given by a density $u_{eq}$. That is, if $u > u_{eq}$ the density of the active population decreases and conversely it decrease whenever   $u < u_{eq},$ according  to an absorption coefficient given by $\alpha$ which may depend on time and space as $u_{eq}$ indeed. The population transport model $u$ which takes into account the congestion related to the maximum value $\rho_{max}$ which is assumed here to be equal to $1$ falls into the scope  of the reaction-diffusion equation \eqref{pdetype+} where $g$ is given by \eqref{absorption0} subject to boundary \eqref{BC0}. Here, $\Gamma_D$ represents  the target boundary region, and $\Gamma_N$ represents    regions where crossing the boundary is not allowed.

\bigskip 
\begin{corollary}\label{CorExample1}
	Assume $0\leq \alpha,\: u_{eq} \in L^\infty(Q)$ and $u_0\in L^\infty(\Omega)$  is such that $0\leq u_0\leq 1$ and $0\leq u_{eq}\leq 1$, a.e. in $\Omega.$ Then the problem 
		\begin{equation} \label{example0}
		\left\{  \begin{array}{ll}\left.
			\begin{array}{l}
				\displaystyle \frac{\partial u }{\partial t}  -\Delta p +\nabla \cdot (u  \: V)= -\alpha  u ( u -  u_{\text{eq}})   \\ \\ 
				\displaystyle  0\leq u\leq 1,\:  p\geq 0,\:  p(u-1)=0 \end{array}\right\}
			\quad  & \hbox{ in } Q \\  \\
			\displaystyle p= 0  & \hbox{ on }\Sigma_D \\  \\
			\displaystyle (\nabla p- u  \: V)\cdot \nu = 0  & \hbox{ on }\Sigma_N\\  \\
			\displaystyle  u (0)=u _0 &\hbox{ in }\Omega,\end{array} \right.
	\end{equation}  
has a unique weak solution : $(u ,p)\in  L^\infty(Q) \times  L^2
	\left(0,T;H^1_D(\Omega)\right)$, 	$  0\leq u\leq 1,$   $p\geq 0,$ $p(u-1)=0$   a.e. in $Q,$  and
	\begin{equation}
		\label{evolexample}
		\displaystyle \frac{d}{dt}\int_\Omega u \:\xi+\int_\Omega ( \nabla p -  u \:V) \cdot  \nabla\xi   =   -   \int_\Omega  \alpha  u ( u -  u_{\text{eq}})  \: \xi  , \quad \hbox{ in }{\D}'(0,T),\quad \forall \: 	\xi\in H^1_D(\Omega).
	\end{equation}
	Moreover, if 
	\begin{equation}\label{condcompress}
		\nabla \cdot V \geq -\alpha ( 1-  u_{\text{eq}}) ,\quad \hbox{ a.e. in }Q, 
	\end{equation}
	then $p\equiv 0$ and $u$ is the unique weak solution of the transport equation 
		\begin{equation} \label{exampletr}
		\left\{  \begin{array}{ll}\left.
			\begin{array}{l}
				\displaystyle \frac{\partial u }{\partial t}   +\nabla \cdot (u  \: V)= -\alpha  u ( u -  u_{\text{eq}})   \\ \\ 
				\displaystyle  0\leq u\leq 1,  \end{array}\right\}
			\quad  & \hbox{ in } Q \\  \\
			\displaystyle  u (0)=u _0 &\hbox{ in }\Omega,\end{array} \right.
	\end{equation}  
	\end{corollary}

\begin{remark}
	Remember that $\nabla \cdot V$ is related in some sense to the compressibility of the velocity field $V$ ; i.e. how much the population is being compressed. Positive divergence, $\nabla \cdot V>0$ means the dynamic  is expanding (individuals are moving away from a point). 	Negative divergence, $\nabla \cdot V<0$,   means the dynamic  is compressing (individuals are moving towards a point, leading to congestion).
	Zero Divergence, $\nabla \cdot V=0,$   the dynamic is incompressible  (the population density would remain constant along flow lines).  This compressibility is the origin of congestion. In cases where this field is given by geodesics towards the exit, this compressibility is related to the position and size of the exit. Thus, condition \eqref{condcompress}  could be interpreted as establishing the link between the compressibility of the velocity field and the capacity of the environment  (or just the neighborhood of the boundary) to fluidize the dynamics or not.  That is how well the environment can handle the flow of individuals  related to the size of the exit, the geometry of the space, or other factors.

	An alternative result that could help improve congestion by adjusting  on the compressibility of $V$ is given in the following corollary. 
	\end{remark}

\bigskip 
\begin{corollary}
Under the assumptions of Corollary \ref{CorExample1}, assume moreover that $ \nabla \cdot (u_{eq}\: V)\in L^1(Q) $, and 	let us consider $u$ the solution of  the problem \eqref{example0}.  If 
\begin{equation}
 \nabla \cdot (u_{eq}\: V)  \geq 0 \quad (\hbox{resp. }   \nabla \cdot (u_{eq}\: V)\leq 0  ,
\end{equation}
and 
\begin{equation}
	u_{0}\leq 	u_{eq}  \quad (\hbox{resp. }  u_{eq}\leq u_0),\quad  \hbox{ a.e. in }\Omega, 
\end{equation}
then \begin{equation}
0\leq u\leq u_{eq}, \quad    (\hbox{resp. }  u_{eq}\leq u \leq 1),\quad  \hbox{ a.e. in }Q.	\end{equation}
  	\end{corollary}
  \begin{proof}
  	This follows easily from \eqref{evolineqcomgp}. Indeed,  in this case we have 
  	 	\begin{eqnarray*} 
  		\frac{d}{dt}	\int_\Omega ( u  -u _{eq})^+ \: dx &\leq&  \int_{[u \geq u _{eq}]}     ( g(.,u)-\nabla \cdot(u_{eq}\: V)))^+ \: dx  \\  \\ 
  		&\leq&  \int_{[u \geq u _{eq}]}     (  -\alpha u  ( u -  u_{\text{eq}})  -\nabla \cdot(u_{eq}\: V)))^+ \: dx  \\  \\ 
  		  		&\leq&0.
  	\end{eqnarray*}
   The converse inequality follows using again \eqref{evolineqcomgp} in the same way. 
    \end{proof}
 
 	\section*{Acknowledgments } This work   was  supported by the CNRST of Morocco under the FINCOM program.  N. Igbida  is very grateful to the EST of Essaouira and its staff for their warm hospitality and the conducive work environment they provided. 
  

 	 \end{document}